\documentclass[11pt, a4paper]{article}

\usepackage{amsmath}
\usepackage{amssymb}
\usepackage{amsthm}
\usepackage{appendix}

\newtheorem{theorem}{Theorem}[section]
\newtheorem{corollary}{Corollary}[section]

\newtheorem{lemma}{Lemma}[section]
\newtheorem{remark}{Remark}[section]

\newtheorem*{thma}{Theorem A}
\newtheorem*{propa}{Proposition A}

\begin{document}
\title{Characterizing Riesz Bases via Biorthogonal Riesz-Fischer sequences}
\author{Elias Zikkos\\
Khalifa University, Abu Dhabi, United Arab Emirates\\
email address:  elias.zikkos@ku.ac.ae and eliaszikkos@yahoo.com}

\maketitle

\begin{abstract}
In this note we prove that if two Riesz-Fischer sequences in a separable Hilbert space $H$
are biorthogonal and one of them is complete in $H$, then both sequences are Riesz bases for $H$.
This complements a recent result by D. T. Stoeva where the same conclusion holds if one replaces
the phrase ``Riesz-Fischer sequences'' by ``Bessel sequences''.
\end{abstract}



\section{Introduction}

Let $H$ be a separable Hilbert space
endowed with an inner product $\langle \,\cdot \, \rangle$ and a norm $||\, \cdot\, ||$.
Let $\{f_n\}_{n=1}^{\infty}$ be a sequence of vectors in $H$. We say that
$\{f_n\}_{n=1}^{\infty}$  is a $\bf{Riesz\,\, basis}$ for $H$ if $f_n=V(e_n)$ where
$\{e_n\}_{n=1}^{\infty}$ is an orthonormal basis for $H$ and $V$ is a bounded
bijective operator from $H$ onto $H$.

One of the many equivalences of Riesz bases \cite[Theorem 1.1]{Stoeva}
states that
\begin{itemize}
\item
A sequence is a Riesz basis for $H$, if and only if it is a complete Bessel sequence 
having a complete biorthogonal Bessel sequence in $H$.
\end{itemize}
Recall that $\{f_n\}_{n=1}^{\infty}$ is a $\bf Bessel$ sequence if there is a positive constant $B$ so that
\[
\sum_{n=1}^{\infty}|\langle f, f_n\rangle|^2< B\cdot ||f||\qquad \forall\,\, f\in H,
\]
and $\{f_n\}_{n=1}^{\infty}$ is $\bf complete$ if its closed span in $H$ is equal to $H$.
$\bf Biorthogonality$ between two sequences $\{f_n\}_{n=1}^{\infty}$ and $\{g_n\}_{n=1}^{\infty}$ means that
\[
\langle f_n, g_m\rangle =\begin{cases} 1, & m=n, \\  0, & m\not=n.\end{cases}
\]

Recently Stoeva \cite{Stoeva} improved the above equivalence by 
assuming completeness on just one sequence.

\begin{thma}\cite[Theorem 2.5]{Stoeva}
Let two sequences in $H$ be biorthogonal.
If both of them are Bessel sequences and one of them is complete in $H$, then
they are Riesz bases for $H$.
\end{thma}

$\bf{Our\,\, goal}$ in this note is to complement Theorem $\bf A$ by replacing
the phrase ``Bessel sequences'' by ``Riesz-Fischer sequences''.
Following Young \cite[Chapter 4, Section 2]{Young}, $\{f_n\}_{n=1}^{\infty}$ is a
$\bf Riesz-Fischer$ sequence in $H$ if the moment problem
\[
\langle f, f_n\rangle =c_n
\]
has at least one solution $f\in H$ for every sequence $\{c_n\}_{n=1}^{\infty}$ in the space $l^2(\mathbb{N})$.
We prove the following.

\begin{theorem}\label{RF}
Let two sequences $\{f_n\}_{n=1}^{\infty}$ and $\{g_n\}_{n=1}^{\infty}$ in $H$ be biorthogonal.
If both of them are Riesz-Fischer sequences and one of them is complete in $H$, then
they are Riesz bases for $H$.
\end{theorem}
The proof is given in Section 3, once we present below some properties of Bessel and Riesz-Fischer sequences
and a nice result connecting the two notions (Proposition $\bf A$).

\section{Riesz-Fischer sequences and Bessel sequences}

In \cite[Chapter 4, Section 2, Theorem 3]{Young} we find the following
two theorems  which provide a necessary and sufficient condition so that
a sequence in $H$ is either a Riesz-Fischer sequence or a Bessel sequence.
Both results are attributed to $\bf{Nina\,\, Bari}$.

\begin{itemize}
\item
$\{f_n\}_{n=1}^{\infty}$ is a Riesz-Fischer sequence in $H$ if and only if
there exists a positive number $A$ so that for any finite scalar sequence
$\{\beta_n\}$ we have
\begin{equation}\label{rieszfischer}
A\sum |\beta_n|^2  \le \left|\left|\sum \beta_n f_n \right|\right|^2.
\end{equation}

\item
$\{f_n\}_{n=1}^{\infty}$ is a Bessel sequence in $H$ if and only if
there exists a positive number $B$ so that for any finite scalar sequence
$\{\beta_n\}$ we have
\begin{equation}\label{bessel}
\left|\left|\sum \beta_n f_n \right|\right|^2 \le B\sum |\beta_n|^2.
\end{equation}
\end{itemize}

If a sequence is both a Bessel sequence and a Riesz-Fischer sequence, then it is called a $\bf Riesz$ sequence
(see Seip \cite[Lemma 3.2]{Seip}). That is, $\{f_n\}_{n=1}^{\infty}$ is
a Riesz sequence if there are some positive constants $A$ and $B$, $A\le B$,
so that for any finite scalar sequence $\{\beta_n\}$ we have
\[
A\sum |\beta_n|^2  \le \bigg|\bigg|\sum \beta_n f_n \bigg|\bigg|^2 \le B\sum |\beta_n|^2.
\]

\begin{remark}
A Riesz sequence is also a Riesz basis for the closure of its linear span in $H$ (see \cite[p. 68]{Christensen}).
Therefore, a complete Riesz sequence in $H$ is a Riesz basis for $H$.
\end{remark}

Now, it easily follows from \eqref{rieszfischer} that a Riesz-Fischer sequence is also a $\bf minimal$ sequence, that is
each $f_n$ does not belong to the closed span of $\{f_k\}_{k\not=n}$ in $H$.

\begin{remark}
It is well known that a sequence is minimal if and only if it has a biorthogonal sequence.
A complete and minimal sequence is called $\bf exact$ and it has a unique biorthogonal sequence.
\end{remark}

Clearly now a Riesz-Fischer sequence has at least one biorthogonal sequence.
As stated in Casazza et al. \cite{Casazza}, one of them is a Bessel sequence (see Green \cite[Proposition 1.2.4]{Green} for a proof).

\begin{propa} \cite[Proposition 2.3, (ii)]{Casazza} 

The Riesz-Fischer sequences in $H$ are precisely the families for which a biorthogonal Bessel sequence exists.
In other words

(Part a) Suppose that a Bessel sequence $\{f_n\}$ is biorthogonal to a sequence $\{g_n\}$ in $H$.
Then $\{g_n\}$ is a Riesz-Fischer sequence.

(Part b) If $\{f_n\}$ is a Riesz-Fischer sequence, then it has a biorthogonal Bessel sequence.

\end{propa}

\section{Proof of Theorem $\ref{RF}$}

First we prove the following result.
\begin{lemma}\label{R}
Let two sequences $\{f_n\}_{n=1}^{\infty}$ and $\{g_n\}_{n=1}^{\infty}$ in $H$ be biorthogonal.
If $\{g_n\}_{n=1}^{\infty}$ is a Riesz sequence and $\{f_n\}_{n=1}^{\infty}$ is complete in $H$, then
$\{g_n\}_{n=1}^{\infty}$ is also complete in $H$, hence both sequences are Riesz bases for $\cal H$.
\end{lemma}
\begin{proof}
Denote by $U$ the closed span of $\{g_n\}_{n=1}^{\infty}$ in $H$.
Since $\{g_n\}_{n=1}^{\infty}$ is a Riesz sequence then it is a Riesz basis for $U$, 
therefore it has a dual biorthogonal Riesz basis for $U$,
call it $\{h_n\}_{n=1}^{\infty}$. Thus
\[
U=\overline{\text{span}}\{h_n\}_{n=1}^{\infty}=\overline{\text{span}}\{g_n\}_{n=1}^{\infty}.
\]
Let $U^{\perp}$ be the orthogonal complement of $U$ in $H$, that is
\[
U^{\perp}=\{f\in H:\,\, \langle f, g \rangle = 0 \quad for\,\, all\,\, g\in U\}.
\]
Hence if $f\in U$, then $\langle f, g_n\rangle = 0$ and $\langle f, h_n\rangle = 0$ for all $n\in\mathbb{N}$.

Now, due to biorthogonality, for fixed $n\in\mathbb{N}$ we have $\langle f_n, g_n\rangle = 1$
and $\langle h_n, g_n\rangle = 1$. We also have $\langle f_n, g_m\rangle = 0$
and $\langle h_n, g_m\rangle = 0$ for all $m\not= n$. Therefore, $\langle (f_n - h_n), g_k\rangle = 0$ for all $k\in\mathbb{N}$,
thus, $(f_n - h_n)$ belongs to $U^{\perp}$.

Suppose that $(f_n - h_n)\not= 0$. Then $(f_n - h_n)$ does not belongs to $U=\overline{\text{span}}\{h_n\}_{n=1}^{\infty}$.
It readily follows that $f_n$ does not belong to $U$ either (if $f_n\in U$ then $(f_n-h_n)\in U$ as well).
Hence $f_n$ belongs to $U^{\perp}$, so
$\langle f_n, g_k\rangle = 0$ for all $k\in\mathbb{N}$, a contradiction since $\langle f_n, g_n\rangle = 1$.
We have now concluded that $f_n=h_n$ and clearly this holds for all $n\in\mathbb{N}$. Therefore,
$\{h_n\}_{n=1}^{\infty}=\{f_n\}_{n=1}^{\infty}$ is a complete Riesz sequence in $H$, hence a Riesz basis for $H$.
The same of course holds for its biorthogonal sequence $\{g_n\}_{n=1}^{\infty}$.
\end{proof}

\begin{corollary}
Suppose that an exponential system $E=\{e^{i\lambda_n t}\}$, with real $\lambda_n$,
is a Riesz-Fischer sequence in $L^2 (-a, a)$, and $E$ has a complete biorthogonal sequence.
Then $E$ is a Riesz basis for $L^2 (-a, a)$.
\end{corollary}
\begin{proof}
It follows from Lindner \cite{Lindner} that $E$ is a Bessel sequence in $L^2 (-a, a)$ as well.  
Hence $E$ is a Riesz sequence in $L^2 (-a, a)$ and from Lemma $\ref{R}$ we obtain the result. 
\end{proof}

Consider now the assumptions of Theorem $\ref{RF}$ and without loss of generality,
suppose that $\{f_n\}_{n=1}^{\infty}$ is complete in $H$, hence it is exact since it is also a minimal sequence.
Thus it has a unique biorthogonal sequence and clearly this is $\{g_n\}_{n=1}^{\infty}$.
Since $\{f_n\}_{n=1}^{\infty}$ is a Riesz-Fischer sequence, then by Proposition $\bf A$ (Part b) it has a biorthogonal Bessel sequence,
and by uniqueness, this is  $\{g_n\}_{n=1}^{\infty}$.
Therefore $\{g_n\}_{n=1}^{\infty}$ is a Bessel sequence and a Riesz-Fischer sequence simultaneously, hence a Riesz sequence.
It then follows from Lemma $\ref{R}$ that $\{f_n\}_{n=1}^{\infty}$ and $\{g_n\}_{n=1}^{\infty}$ are Riesz bases for $H$.
The proof of Theorem $\ref{RF}$ is now complete.

\end{document}